\documentclass[11pt,reqno]{amsart}

\usepackage{amsmath, amssymb, amsthm}
\usepackage{color}
\usepackage[pdftex, colorlinks=true, urlcolor=blue, linkcolor=blue, citecolor=blue]{hyperref}
\usepackage{hyperref}
\usepackage{xcolor}

\usepackage[margin=1.1in]{geometry}
\newtheorem{theorem}{Theorem}[section]
\newtheorem{lemma}{Lemma}[section]
\newtheorem{proposition}{Proposition}[section]
\newtheorem{definition}{Definition}[section]
\newtheorem{corollary}{Corollary}[section]

\newtheorem{remark}{Remark}[section]
\newtheorem{example}{Example}[section]
\newtheorem*{ack}{Acknowledgments}

\everymath{\displaystyle}

\begin{document}
\title[Generalized Alexandrov theorems in spacetimes with integral conditions]{Generalized Alexandrov theorems in spacetimes with integral conditions}
\author[K.-K. Kwong]{Kwok-Kun Kwong}
\address{School of Mathematics and Applied Statistics, University of Wollongong, NSW 2522, Australia}
\email{\href{mailto:kwongk@uow.edu.au}{kwongk@uow.edu.au}}

\author[X. Wang]{Xianfeng Wang}
\address{School of Mathematical Sciences and LPMC, Nankai University,
Tianjin 300071, P.R. China. }
\email{\href{mailto:wangxianfeng@nankai.edu.cn}{wangxianfeng@nankai.edu.cn}}
\begin{abstract}
We investigate integral conditions involving the mean curvature vector $\vec{H}$ or mixed higher-order mean curvatures, to determine when a codimension-two submanifold $\Sigma$ lies on a shear-free (umbilical) null hypersurface in a spacetime. We generalize the Alexandrov-type theorems in spacetime introduced in \cite{wang2017Minkowski} by relaxing the curvature conditions on $\Sigma$ in several aspects. Specifically, we provide a necessary and sufficient condition, in terms of a mean curvature integral inequality, for $\Sigma$ to lie in a shear-free null hypersurface. A key component of our approach is the use of Minkowski formulas with arbitrary weight, which enables us to derive rigidity results for submanifolds with significantly weaker integral curvature conditions.
\end{abstract}

\maketitle

\section{Introduction}
The study of spacetime geometries has been a central topic of interest in the field of general relativity. Understanding the causal future or past of a given geometric object is crucial in providing insights into the nature of the surrounding spacetime. Codimension-two submanifolds play a significant role in general relativity, with their null expansions intricately connected to gravitational energy, as demonstrated in Penrose's singularity theorem \cite{penrose1965gravitational}. One of the key problems in this field is determining when a given codimension-two submanifold lies within the null hypersurface generated by a ``round sphere''. Such a null hypersurface is referred to as ``shear-free'' (Definition \ref{shear free}), and is analogous to an umbilical hypersurface in Riemannian geometry. This problem is of physical significance as it has implications for our understanding of the geometry of spacetime.

Moreover, this problem can be regarded as an extension of the problem of characterizing round spheres in some Riemannian manifolds, which are collectively referred to as Alexandrov-type theorems. Indeed, since Euclidean space and hyperbolic space are both spacelike hypersurfaces in the Minkowski spacetime, it is natural to ask whether the Alexandrov theorems in these two space forms can be obtained by a more general rigidity result concerning codimension-two submanifolds in the Minkowski or other more general spacetimes. Answering this question could yield a more unified approach to understanding the geometry of different space forms.

In this paper, we investigate the conditions on the mean curvature vector $\vec{H}$ or mixed higher order mean curvatures that ensure a codimension-two submanifold $\Sigma$ lies in a shearfree null hypersurface. Our results generalize a number of the Alexandrov type theorems in spacetime introduced in \cite{wang2017Minkowski}, by relaxing the curvature conditions on $\Sigma$.
Specifically, one of our main results, Theorem \ref{alex}, provides a necessary and sufficient condition in terms of the mean curvature vector, to ensure that $\Sigma$ lies in a shear-free null hypersurface. These results have applications in the study of the null geometry of submanifolds, which is relevant to the behavior of physical systems in curved spacetimes, such as the divergence of light rays emanating from a submanifold. Moreover, we believe that the weighted Minkowski formulas developed here provide a new tool for future studies on curved spacetimes with symmetry.
We notice that in \cite{hijazi2019Alexandrov},  the results in \cite{wang2017Minkowski} were also generalized by relaxing the assumption on the incoming null smoothness of $\Sigma$.

Our work is motivated by two sources in addition to \cite{wang2017Minkowski}. The first is the weighted Minkowski formulas \cite{kwong2016extension, kwong2018weighted} developed by Lee, Pyo, and the first author. These formulas allow for the introduction of an arbitrary weight to the classical Minkowski formulas, giving more flexibility to their application. For example, the weighted Minkowski formulas can be used to prove an Alexandrov type theorem in standard space forms \cite{kwong2018weighted}, with the assumption that the mean curvature is a non-increasing function in the radial distance $r$ from a fixed point, rather than being constant.
The second source of motivation is the recent work by Gao and Ma \cite{gao2021characterizations} on the proof of Jellet-Liebmann and Alexandrov type theorems. In their work, among other things, the authors use an inequality assumption involving the gradient of the mean curvature $H$ to prove an Alexandrov-type theorem, instead of depending on the constancy of $H$. Both \cite{kwong2016extension, kwong2018weighted} and \cite{gao2021characterizations} aim to relax the requirement of constant mean curvature or higher order mean curvatures for hypersurfaces, which guarantee their roundness.

Compared to \cite{wang2017Minkowski}, our assumptions are much weaker in several aspects. To see this, let us first state one of our main results.
\begin{theorem}[Theorem \ref{alex}]\label{thm intro1}
Let $V$ be a spherically symmetric spacetime as in \nameref{assA} and $\Sigma$ be a
closed embedded, spacelike codimension-two submanifold in $V$ such that $\vec H$ is spacelike.
Assume $\Sigma$ is future incoming null embedded. Then $\Sigma$ satisfies
\begin{equation}\label{cond intro}
\int_\Sigma \frac{1}{|\vec{H}|^{2}} Q\left(\nabla(\log |\vec{H}|)-\sum_{a=1}^{n-1} \alpha_{\vec{H}}\left(e_{a}\right) e_{a}, \vec{H}+\vec{J}\right)d\mu\le 0
\end{equation}
if and only if $\Sigma$ lies in an incoming shear-free null hypersurface.
\end{theorem}
Here, $Q$ is a conformal Killing-Yano two-form, $\alpha_{\vec H}$ is the connection one-form in the mean curvature gauge, $\{e_a\}_{a=1}^{n-1}$ is a local orthonormal frame on $\Sigma$, and $\vec{J}$ is the reflection of $\vec{H}$ along the incoming light cone. The precise definitions of these terms will be given in Section \ref{sec mink}.
Important examples of spacetimes satisfying \nameref{assA} include the exterior Schwarzschild spacetime and the anti-deSitter spacetime, both with $Q=r d r \wedge d t$ in suitable coordinates.

To understand why our condition \eqref{cond intro} is weaker compared to \cite{wang2017Minkowski} and why it is physically relevant, let us focus on the conditions given in \cite[Theorem 3.14]{wang2017Minkowski} for a codimension-two submanifold $\Sigma$ in a static (cf. \cite[p. 119]{wald} for definition) spherically symmetric spacetime to lie in a shear-free null hypersurface.
The main condition in \cite[Theorem 3.14]{wang2017Minkowski} is that for a future incoming null normal vector field $\underline{L}$ along the submanifold $\Sigma$ such that $\langle\vec{H}, \underline{L}\rangle$ is a positive constant and $(D \underline{L})^{\perp}=0$ on $\Sigma$. This can also be written as the pointwise condition $d \log |\vec{H}|-\alpha_{\vec{H}}=0$. This condition can be viewed as the Lorentzian counterpart of the assumption of triviality of the normal bundle of a hypersurface and the constant mean curvature condition $H=\mathrm{constant}$ in Riemannian geometry.
In contrast, our condition \eqref{cond intro} is an integral inequality, and is thus much weaker than any pointwise constant condition. Moreover, our formulation involves the quantity $\nabla(\log |\vec{H}|)-\sum_{a=1}^{n-1} \alpha_{\vec{H}}\left(e_a\right) e_a$, which we believe is natural as it does not rely on the choice of a null vector field, which is defined only up to scaling by a positive function. We will provide the precise definitions of each term in the next section.

Before we delve into the idea of the proof of the main results, let us take a moment to consider the advantages of our assumption from a physics perspective. Our assumption represents both an integral and an inequality condition, which are more practical and stable than a pointwise constant condition. This is because it is infeasible to measure every point on a submanifold, and there is always some imprecision in measurement. In practice, measurements can only be taken by sampling points. The integral in \eqref{cond intro} allows for approximation by taking measurements over a sufficiently large number of sample points. Moreover, the inequality condition in \eqref{cond intro} allows room for imprecision, making it more physically relevant. Thus, our result not only generalizes previous work but also provides a more stable and physically relevant condition.

Let us now explain the idea of our results. To prove Theorem \ref{thm intro1}, we will need a weighted version of the spacetime Minkowski formula. This formula extends \cite[Theorem A]{wang2017Minkowski} by allowing an arbitrary weight on the curvature integral, and can also be regarded as the Lorentzian version of \cite[Theorem 1.1]{kwong2016extension}:
\begin{theorem}[Theorem \ref{mink2}]\label{thm intro2}
Let $\Sigma$ be a closed immersed oriented spacelike codimension-two submanifold in an $(n+1)$-dimensional Riemannian or Lorentzian manifold $V$ that possesses a conformal Killing-Yano two-form $Q$. Assume the mean curvature vector $\vec H$ of $\Sigma$ is spacelike. Then for any smooth function $f$ defined on $\Sigma$, we have
\begin{equation}\begin{split} \label{spacetime mink 1}
&\frac{n-1}{n} \int_{\Sigma} \frac{f}{|\vec{H}|}\left\langle\xi, \pm\vec{H}+\vec{J}\right\rangle d \mu+\int_{\Sigma} \frac{f}{|\vec{H}|} Q\left(\vec{H}, \vec{J}\right) d \mu+\int_{\Sigma} \frac{1}{|\vec{H}|} Q\left(\nabla f \pm f \sum_{a=1}^{n-1}\alpha_{\vec{H}}\left(e_{a}\right) e_{a}, \pm\vec{H}+\vec{J}\right) d \mu\\
&=0.
\end{split}\end{equation}
\end{theorem}
This formula is the Lorentzian version of the following weighted Minkowski formula \cite{kwong2016extension} in $\mathbb R^{n}$: for any smooth function $f$ on a closed hypersurface $\Sigma$, we have
\begin{equation}\begin{split}\label{rn mink}
\int_\Sigma f d\mu =\int_\Sigma f H_1\langle X, \nu\rangle d\mu -\frac{1}{n-1}\int_\Sigma\langle \nabla f, X\rangle d\mu,
\end{split}\end{equation}
where $X$ is the position vector, $H_1$ is the normalized mean curvature, and $\nu$ is the unit outward normal. In fact, if $V$ is the Minkowski spacetime and $\Sigma$ lies in a spacelike hyperplane, then \eqref{spacetime mink 1} is reduced to \eqref{rn mink}.

The next ingredient is a spacetime Heintze-Karcher type inequality developed in \cite{wang2017Minkowski}. For our purpose, we prefer to write it in the following form.
\begin{theorem}[Theorem \ref{HK ineq2}, {\cite[Theorem 3.12]{wang2017Minkowski}}]\label{thm intro3}
Suppose $\Sigma$ is a future incoming null embedded spacelike codimension-two submanifold satisfying the assumption in Theorem \ref{thm intro1}. Then we have
\begin{align}\label{hk spacetime}
-(n-1) \int_{\Sigma} \frac{1}{|\vec{H}|^{2}} \left\langle\frac{\partial }{\partial t }, \vec{H}+\vec{J}\right\rangle d \mu+ \int_{\Sigma} \frac{1}{|\vec{H}|^{2}} Q(\vec{H}, \vec{J}) d \mu \ge 0.
\end{align}
Moreover, the equality holds if and only if $\Sigma$ lies in an incoming shear-free null hypersurface.
\end{theorem}
This is a far-reaching generalization of the following Heintze-Karcher inequality in $\mathbb R^{n}$: for a smooth embedded mean-convex hypersurface $\Sigma$ enclosing $\Omega$, we have
\begin{equation}\label{hk intro}
\int_{\Sigma} \langle X, \nu\rangle d\mu \le \int_{\Sigma}\frac{1}{H_1} d\mu.
\end{equation}
The equality holds if and only if $\Sigma$ is a round sphere.

In order to facilitate the discussion and illustrate the idea how Theorem \ref{thm intro2} and Theorem \ref{thm intro3} can be used to obtain Theorem \ref{thm intro1}, we limit our focus to the case where the ambient space is $\mathbb R^n$. This can be regarded as a special case of Theorem \ref{thm intro1} in which $\Sigma$ lies on a spacelike hyperplane in $V=\mathbb R^{n, 1}$. In this case, the condition that corresponds to \eqref{cond intro} is
\begin{equation}\label{cond rn}
\int_{\Sigma} \frac{1}{H_1^2}\langle \nabla H_1, X\rangle d\mu\le 0.
\end{equation}
To see why this implies $\Sigma$ is round, one can substitute $f=\frac{1}{H_1}$ into \eqref{rn mink} to obtain
\begin{align}\label{intro rigid}
\int_\Sigma \frac{1}{H_1}d\mu =\int_{\Sigma}\langle X, \nu\rangle d\mu +\frac{1}{n-1}\int_{\Sigma}\frac{1 }{H_1^2} \langle \nabla H_1, X\rangle d\mu \le \int_{\Sigma}\langle X, \nu\rangle d\mu.
\end{align}
The rigidity case of \eqref{hk intro} then implies that $\Sigma$ is round.

Moreover, we see from \eqref{intro rigid} that if \eqref{cond rn} is false, i.e. $\int_{\Sigma} \frac{1}{H_1^2}\left\langle\nabla H_1, X\right\rangle d \mu > 0$, then \eqref{hk intro} is strict. Therefore $\Sigma$ is non-round. So we have the following characterization of a round sphere.
\begin{theorem}
The condition
$$
\int_{\Sigma} \frac{1}{H_1^2}\left\langle\nabla H_1, X\right\rangle d \mu \le 0
$$
is a necessary and sufficient condition for a closed embedded mean-convex hypersurface $\Sigma$ in $\mathbb R^n$ to be round.
\end{theorem}
\begin{remark}
We notice that there is an integral condition in a recent paper \cite[Remark 6.1]{li2022locally}, which is similar to \eqref{cond intro} or \eqref{cond rn}, that characterizes closed totally umbilical (hence round) spacelike hypersurfaces in the de Sitter space. However, this characterization is based on a conjectured Heintze-Karcher type rigidity in the de Sitter space, and the submanifold is assumed to be a hypersurface rather than a codimension-two one.
\end{remark}

At this point, let us explain the rationale behind assuming an integral inequality condition instead of a pointwise condition. This is due to our use of the rigidity case of the Heintze-Karcher inequality (\eqref{hk spacetime} or \eqref{hk intro}), which is an integral inequality on $\Sigma$. Hence, we should expect a condition that is less restrictive than a pointwise constant condition. In this sense, our integral inequality condition is natural.

This line of reasoning can be readily applied to the Lorentzian case under the assumptions in Theorem \ref{thm intro1}, thereby establishing our main result. We remark that the integral condition \eqref{cond rn} is less restrictive than the pointwise condition $\langle\nabla H_1, X\rangle \le 0$ stated in \cite[Theorem 1]{gao2021characterizations}. With the same techniques, we are able to generalize both Theorem 1 and Theorem 2 in the warped product manifold settings considered in \cite{gao2021characterizations} by substituting the pointwise condition on the higher order mean curvatures with an integral condition. However, since the main objective of this paper is to study spacetime Alexandrov theorems, we will not delve into these extensions in detail.

We will also prove similar Alexandrov-type results for mixed higher order mean curvatures in spacetime with constant curvature.
Let us first briefly explain the relevant notations. Consider a null frame $L$ and $\underline L$ of $\Sigma$ that satisfy the condition $\langle L, \underline{L}\rangle=-2$. Corresponding to this null frame, we can define the connection one-form $\zeta=\zeta_L$ and two null second fundamental forms $\chi$ and $\underline{\chi}$. Remarkably, the two-form $d\zeta_L$ is independent of the choice of $L$ and $\underline{L}$ (Lemma \ref{lem: d zeta}). Consequently, a natural cohomology condition on $\zeta$ ensures the existence of a canonical null frame $L$, $\underline L$, with respect to which $\Sigma$ is torsion-free (i.e. $\zeta_L=0$). Parallel to the theory of hypersurfaces in the Euclidean space, we can define the \textit{mixed} higher order mean curvature $P_{r, s}(\chi, \underline{\chi})$, as well as the \textit{mixed} Newton tensors $T_{r, s}(\chi,\underline{\chi})$ and $\underline{T}_{r, s}(\chi,\underline{\chi})$, with respect to $\chi$ and $\underline{\chi}$ for a codimension-two spacelike submanifold $\Sigma$ in spacetime. For precise definitions, please refer to Section \ref{sec higher}.

We will show that the Alexandrov-type result can be extended for mixed higher order mean curvatures in a spacetime with constant curvature. Besides imposing an integral inequality condition instead of a pointwise constant condition, we also relax the zero torsion condition in \cite[Theorem 5.1]{wang2017Minkowski} to a cohomology condition. We consider this cohomology condition to be a more natural choice since it remains invariant under any change of the null frame. We will prove the following theorem.
\begin{theorem}[Theorem \ref{thm higher order alex}]\label{intro thm 1.5}
Let $\Sigma$ be a past incoming null embedded (see Definition \ref{def null embed}), closed spacelike codimension-two submanifold in an $(n+1)$ dimensional spacetime of constant curvature
such that $d \zeta=0$ and $[\zeta]=0 \in H_{d R}^1(\Sigma)$.
Then there exists a null frame $L$ and $\underline{L}$ (given by Proposition \ref{prop torsion}) with respect to which $\Sigma$ is torsion free. Moreover, assume that the second fundamental form w.r.t. $L$ and $\underline L$ satisfies $\chi \in \Gamma_{r}$.
If
\begin{align*}
\int_{\Sigma} \frac{1}{P_{r, 0}^2}Q\left(T_{r, 0}(\nabla P_{r, 0}), L\right) d \mu\le 0,
\end{align*}
then $\Sigma$ lies in an outgoing shear-free null hypersurface.
\end{theorem}
There is a result corresponding to the future incoming case as well, see Theorem \ref{thm higher order alex'}.

The proof of weighted Minkowski formulas for higher order mean curvatures relies heavily on the divergence-free property of Newton tensors, which is valid in Lorentzian space forms as discussed in Section \ref{sec higher}. However, this property does not hold in more general spacetimes such as the Schwarzschild spacetime. In such cases, the divergence of the Newton tensors involves the ambient curvature, making it more difficult to analyze. Nonetheless, Brendle-Eichmair \cite{BE2013} discovered an interesting observation that, under certain natural conditions on the submanifold $\Sigma$ in the Riemannian Schwarzschild manifold, the divergence term in the Minkowski formula has a positive sign. This observation enables them to obtain Minkowski-type inequalities and prove Alexandrov theorems in this setting. A similar observation was made in the Schwarzschild spacetime in \cite[Lemma 6.1]{wang2017Minkowski}. By combining this observation with the weighted version of Minkowski inequalities, which we derive in Section \ref{sec Schwarzschild}, we can obtain integral Alexandrov-type theorems in the Schwarzschild spacetime (Theorem \ref{thm schwarzschil alex}).

The remainder of this paper is structured as follows. In Section \ref{sec mink}, we derive weighted spacetime Minkowski formulas for a closed immersed oriented spacelike codimension-two submanifold in a Lorentzian manifold $V$ possessing a conformal Killing-Yano two-form. We formulate these formulas in a manner that will be useful in later sections (see Theorem \ref{mink2}). In Section \ref{sec integral Alexandrov type theorem}, we provide a proof for Theorem \ref{thm intro1}, which presents a necessary and sufficient condition on a certain integral involving mean curvature. This condition ensures that a closed embedded spacelike codimension-two submanifold lies in an incoming shear-free null hypersurface. In Section \ref{sec higher}, we introduce the mixed higher order mean curvatures $P_{r, s}$ of a codimension-two spacelike submanifold in an $(n+1)$-dimensional spacetime $V$, along with corresponding weighted Minkowski formulas (Theorem \ref{thm higher mink}) on $\Sigma$. We then present a proof for Theorem \ref{intro thm 1.5}, which provides an integral Alexandrov type theorem for higher order mean curvatures under a natural cohomology condition on $\Sigma$ that does not rely on a null frame being chosen a priori. In Section \ref{sec Schwarzschild}, we prove some generalized Alexandrov theorems in the Schwarzschild spacetime, subject to natural assumptions on $\Sigma$. Although the ambient curvature is not constant and the mixed Newton tensors are no longer divergence-free in the Schwarzschild case, we can still derive a weighted Minkowski-type inequality under certain conditions on $\Sigma$. This observation is similar to the one made in \cite[Proposition 8]{BE2013}, which allowed for a Minkowski-type inequality to be obtained in the Riemannian Schwarzschild manifold.

\begin{ack}
Kwok-Kun Kwong was supported by grant FL150100126 of the Australian Research Council. Xianfeng Wang was supported by NSFC Grant No. 11971244 and ``the Fundamental Research Funds for the Central Universities" Grant No. 050-63233068. The first author would like to thank Professors Mu-Tao Wang and Ye-Kai Wang for their valuable discussions.
\end{ack}

\section{Weighted spacetime Minkowski formulas}\label{sec mink}
Let $F: \Sigma^{n-1} \rightarrow\left(V^{n+1}, \langle \cdot, \cdot \rangle\right)$, be a closed immersed oriented spacelike codimension-two submanifold in an oriented $(n+1)$-dimensional Lorentzian manifold $\left(V^{n+1}, \langle \cdot, \cdot \rangle\right)$. Let $\vec{H}$ denote the mean curvature vector of $\Sigma$.
We assume the normal bundle is also orientable and let $\{e_a\}_{a=1}^{n-1}$ be a local orthonormal frame of $\Sigma$, with $\{\omega^a\}_{a=1}^{n-1}$ being its dual coframe.
Let $D$ and $\nabla$ denote the Levi-Civita connection of $V$ and $\Sigma$, respectively. In this paper, $V$ and $\Sigma$ are assumed to be connected.

A two-form $Q$ on $V$ is said to be a conformal Killing-Yano two-form \cite[Definition 1]{jezierski2006conformal} if
$$
\left(D_{X} Q\right)(Y, Z)+\left(D_{Y} Q\right)(X, Z)
=\frac{2}{n}\left(\langle X, Y\rangle\langle\xi, Z\rangle-\frac{1}{2}\langle X, Z\rangle\langle\xi, Y\rangle-\frac{1}{2}\langle Y, Z\rangle\langle\xi, X\rangle\right)
$$
for any tangent vectors $X, Y$ and $Z$, where $\xi=\mathrm{div}_{V} Q$.

Conformal Killing-Yano two-forms are often referred to as ``hidden symmetries'' in physics, as they do not necessarily correspond to any continuous symmetry of the ambient space, unlike conformal Killing vector fields.
\begin{remark}
If $(M^n, g)$ is a Riemannian manifold which possesses a conformal Killing 1-form $ Q$, i.e., $\nabla^M_X Q(Y)+\nabla^M_Y Q(X)= \frac{2}{ n } \mathrm{div}_M (Q) g(X, Y)$ for any tangent vectors $X$ and $Y$, then it is not hard to see that on $(M\times \mathbb R, -dt^2+g)$, $Q\wedge dt$ is a conformal Killing-Yano two-form.
\end{remark}

The following is the weighted version of the spacetime Minkowski formula \cite[Theorem 2.2]{wang2017Minkowski}, see also \cite{kwong2016extension, kwong2018weighted}. We notice that this is indeed equivalent to \cite[Theorem 2.2]{wang2017Minkowski} by the change of null vector $\underline L\to f \underline L$. However, we choose to allow an arbitrary weight $f$ in the weighted version of the spacetime Minkowski formula, as we plan to rewrite this formula in a more suitable form for our later purposes.

\begin{proposition}
Let $\Sigma$ be a closed immersed oriented spacelike codimension-two submanifold in an $(n+1)$-dimensional Riemannian or Lorentzian manifold $V$ that possesses a conformal Killing-Yano two-form $Q$. For any null normal vector field $\underline{L}$ of $\Sigma$ and a smooth function $f$ defined on $\Sigma$, we have
\begin{align}\label{2.2}
\frac{n-1}{n} \int_{\Sigma}f\langle\xi, \underline{L}\rangle d \mu+\int_{\Sigma}f Q(\vec{H}, \underline{L}) d \mu+\int_{\Sigma} f Q\left(e_a, \left(D^{a} \underline{L}\right)^{\perp}\right) d \mu+\int_\Sigma Q\left(\nabla f, \underline{L}\right)d\mu=0
\end{align}
where $\xi=\mathrm{div}_{V} Q$.
\end{proposition}
\begin{proof}
Consider the one-form $\mathcal Q =Q\left(e_a, \underline{L}\right) \omega^{a}$
on $\Sigma$. By the proof of \cite[Theorem 2.2]{wang2017Minkowski},
\begin{align*}
\mathrm{div}_{\Sigma} \mathcal Q
& =\frac{n-1}{n}\langle\xi, \underline{L}\rangle+Q(\vec{H}, \underline{L})+Q\left(e_a, \left(D^{a} \underline{L}\right)^{\perp}\right)
\end{align*}
and therefore
\begin{align*}
\mathrm{div}_{\Sigma} \left(f\mathcal Q\right)
& =\frac{n-1}{n}f\langle\xi, \underline{L}\rangle+fQ(\vec{H}, \underline{L})+fQ\left(e_a, \left(D^{a} \underline{L}\right)^{\perp}\right)+Q\left(\nabla f, \underline{L}\right).
\end{align*}
The assertion follows by integrating this over $\Sigma$.
\end{proof}

In the case of the Schwarzschild spacetime, we take $Q=r d r \wedge d t$, then $\xi=-n \frac{\partial}{\partial t}$.
In particular, \cite[Theorem A]{wang2017Minkowski} can be extended to
\begin{theorem}
Consider the two-form $Q=r d r \wedge d t$ on the Schwarzschild spacetime. For a closed oriented spacelike codimension-two submanifold $\Sigma$ in the Schwarzschild spacetime and a null normal vector field $\underline{L}$ along $\Sigma$ and a smooth function $f$ defined on $\Sigma$, we have
\begin{align*}
-(n-1) \int_{\Sigma}f\left\langle\frac{\partial}{\partial t}, \underline{L}\right\rangle d \mu+\int_{\Sigma} fQ(\vec{H}, \underline{L}) d \mu+\sum_{a=1}^{n-1} \int_{\Sigma}f Q\left(e_{a}, \left(D_{e_{a}} \underline{L}\right)^{\perp}\right) d \mu+\int_{\Sigma} Q\left(\nabla f, \underline{L}\right) d \mu=0.
\end{align*}
\end{theorem}

Recall that $\vec{H}$ denotes the mean curvature vector of $\Sigma$.
Let $\left\{e_{n}, e_{n+1}\right\}$ be an oriented orthonormal frame of the normal bundle such that $e_{n}$ is spacelike and $e_{n+1}$ is future timelike. Define the normal vector field $\vec{J}$ by reflecting $\vec{H}$ along the incoming light cone:
$$ \vec{J}=\left\langle\vec{H}, e_{n+1}\right\rangle e_{n}-\left\langle\vec{H}, e_{n}\right\rangle e_{n+1}. $$
Then $\vec{J}$ satisfies $\langle\vec{J}, \vec{J}\rangle=-\langle\vec{H}, \vec{H}\rangle$ and $\langle\vec{J}, \vec{H}\rangle=0. $ In fact, $\vec{J}$ is uniquely characterized by these properties up to a sign.

Now, suppose the mean curvature vector is spacelike everywhere. Physically, $\Sigma$ is called an untrapped submanifold. This is because $\vec{H}$ is a causal vector when $\Sigma$ is trapped, with negative null expansions ($-\langle \vec H, L\rangle<0$ and $-\langle \vec H, \underline{L} \rangle<0$) for a future outgoing (resp. incoming) null normal $L$ (resp. $\underline{L}$), or when $\Sigma$ is marginally trapped, having both null expansions equal to zero.
We take $e_{n}^{\vec{H}}=-\frac{\vec{H}}{|\vec{H}|}$ and $e_{n+1}^{\vec{H}}=\frac{\vec{J}}{|\vec{H}|}$ and write $\alpha_{\vec{H}}$ for the connection one-form with respect to this mean curvature vector \cite{wang2017Minkowski}:
$$ \alpha_{\vec{H}}(V)=\left\langle D_{V} e_{n}^{\vec{H}}, e_{n+1}^{\vec{H}}\right\rangle \textrm{ for $V \in T \Sigma$. }$$

For our purpose, we rewrite the weighted Minkowski formula in the following form.
\begin{theorem}\label{mink2}
Let $\Sigma$ be a closed immersed oriented spacelike codimension-two submanifold in an $(n+1)$-dimensional Riemannian or Lorentzian manifold $V$ that possesses a conformal Killing-Yano two-form $Q$. Assume the mean curvature vector $\vec H$ of $\Sigma$ is spacelike. Then for any smooth function $f$ defined on $\Sigma$, we have
\begin{equation}\begin{split} \label{2.2'}
&\frac{n-1}{n} \int_{\Sigma} \frac{f}{|\vec{H}|}\left\langle\xi, \vec{H}+\vec{J}\right\rangle d \mu+\int_{\Sigma} \frac{f}{|\vec{H}|} Q\left(\vec{H}, \vec{J}\right) d \mu+\int_{\Sigma} \frac{1}{|\vec{H}|} Q\left(\nabla f + f \sum_{a=1}^{n-1}\alpha_{\vec{H}}\left(e_{a}\right) e_{a}, \vec{H}+\vec{J}\right) d \mu\\
&=0
\end{split}\end{equation}
and
\begin{equation}\begin{split}\label{2.2''}
&\frac{n-1}{n} \int_{\Sigma} \frac{f}{|\vec{H}|}\left\langle\xi, -\vec{H}+\vec{J}\right\rangle d \mu+\int_{\Sigma} \frac{f}{|\vec{H}|} Q\left(\vec{H}, \vec{J}\right) d \mu+\int_{\Sigma} \frac{1}{|\vec{H}|} Q\left(\nabla f-f \sum_{a=1}^{n-1} \alpha_{\vec{H}}\left(e_{a}\right) e_{a}, -\vec{H}+\vec{J}\right) d \mu\\
&=0.
\end{split}\end{equation}
\end{theorem}

\begin{proof}
Define the two null vectors $\underline{L}=\frac{1}{|\vec{H}|}(\vec{H}+\vec{J}) $ and $L= \frac{1}{|\vec{H}|}(-\vec{H}+\vec{J})$. Then $Q(\vec{H}, \underline{L})=\frac{1}{|\vec{H}|} Q(\vec{H}, \vec{J})$ and
it follows from the computation in \cite[Proposition 3.4]{wang2017Minkowski} that
\begin{equation*}\begin{split}
\left(D_{e_a} \underline{L}\right)^{\perp}= - \frac{1}{2}\left\langle D_{e_a} \underline{L}, L\right\rangle \underline{L}= \frac{\alpha_{\vec H }(e_a)} {|\vec{H}|} \left(\vec{H} + \vec{J} \right)
\end{split}\end{equation*}
and
\begin{equation*}\begin{split}
\left(D_{e_a} L\right)^{\perp}=-\frac{1}{2}\left\langle D_{e_a} L, \underline L\right\rangle L=-\frac{\alpha_{\vec{H}}\left(e_{a}\right)}{|\vec{H}|}\left(-\vec{H}+\vec{J}\right).
\end{split}
\end{equation*}
The formula \eqref{2.2} then becomes
\begin{align*}
&\frac{n-1}{n} \int_{\Sigma}\frac{f}{|\vec{H}|}\left\langle\xi, \vec{H}+\vec{J} \right\rangle d \mu+\int_{\Sigma}\frac{f}{|\vec{H}|} Q\left(\vec{H}, \vec J\right) d \mu
+\int_{\Sigma} \frac{1}{|\vec{H}|}Q\left(\nabla f + f\sum_{a=1}^{n-1}\alpha_{\vec{H}}\left(e_{a}\right) e_{a}, \vec H+\vec J\right) d \mu\\
&=0.
\end{align*}
The formula \eqref{2.2''} can be similarly obtained.
\end{proof}
\paragraph{\textbf{Assumption A}}\label{assA}
We assume $V$ is an $(n+1)$-dimensional spacetime that satisfies the null convergence condition, i.e.,
$\mathrm{Ric}(L, L) \ge 0 $ for any null vector $L $,
and the metric $\bar{g}$ on $V= \mathbb R \times M$ is of the form
\begin{equation*}
\bar{g}=-f^{2}(r) d t^{2}+\frac{1}{f^{2}(r)} d r^{2}+r^{2} g_{N}
\end{equation*}
where $M=\left[r_{1}, r_{2}\right) \times N$ equipped with metric
\begin{equation}\label{3.10}
g=\frac{1}{f^{2}(r)} d r^{2}+r^{2} g_{N}
\end{equation}
and $\left(N, g_{N}\right)$ is a compact $n$-dimensional Riemannian manifold. We consider two cases:\\
(i) $f:[0, \infty) \rightarrow \mathbb R$ with $f(0)=1, f^{\prime}(0)=0$, and $f(r)>0$ for $r \ge 0$\\
(ii) $f:\left[r_{0}, \infty\right) \rightarrow \mathbb R$ with $f\left(r_{0}\right)=0$ and $f(r)>0$ for $r>r_{0}$.

In case (i), $(V, \bar{g})$ is complete. In case (ii), $V$ contains an event horizon $H =\left\{r=r_{0}\right\}$. Note that the warped product manifolds considered in \cite{brendle2013constant} are embedded as totally geodesic slices in these spacetimes.

\begin{example}\label{example1}
\begin{enumerate}
\item
An example satisfying the \nameref{assA} is the $(n+1)$-dimensional Schwarzschild spacetime with mass $m \ge 0$, which is equipped with the metric
\begin{align}\label{1.2}
\bar{g}=-\left(1-\frac{2 m}{r^{n-2}}\right) d t^{2}+\frac{1}{1-\frac{2 m}{r^{n-2}}} d r^{2}+r^{2} g_{\mathbb S^{n-1}}, \quad r^{n-2}>2 m.
\end{align}
It is the unique spherically symmetric spacetime that satisfies the vacuum Einstein equations. It can be seen that $Q=r d r \wedge d t$ is a conformal Killing-Yano two-form on the Schwarzschild spacetime.
\item
The anti-deSitter spacetime (\cite[Sec. 5]{chen2019Minkowski}) with metric
\begin{equation}\label{ads}
\overline g=-\left(1+r^{2}\right) d t^{2}+\frac{1}{1+r^{2}} d r^{2}+r^{2} g_{\mathbb S^{n-1}}, \quad r\ge0,
\end{equation}
also admits the conformal Killing-Yano two-form $Q=r d r \wedge d t$.

As explained in \cite[Lemma 3.8 and p. 261]{wang2017Minkowski} and \cite[Proposition 2.1 and Section 5]{brendle2013constant}, this metric also satisfies \nameref{assA}, as well as some other conditions in \cite{brendle2013constant}.
For the convenience of the reader, we sketch the calculation here.
Brendle \cite{brendle2013constant} writes the metric in geodesic coordinates
$$ d \bar{r}^2+h^{2}(\bar{r}) g_{N}. $$
It is equivalent to \eqref{3.10} by a change of variables $r(\overline r)=h(\overline r)$ and $\sqrt{1+r^2}=f=\frac{d h}{d \bar{r}}$. The conditions (H1) and (H2) in \cite{brendle2013constant} are easily checked
except one of the conditions in (H1) that $h^{\prime \prime}(0)>0$, which is equivalent to $f(0)f'(0)>0$. However, we can see that $f(r)f'(r)=r>0$ on $(0, \infty)$ and the condition $h''(0)>0$ is used only in \cite[Proposition 2.3]{brendle2013constant} to ensure that there exists $(0, \overline r_1)$ on which $h''>0$, which is still true in this case.

The condition (H3) in \cite{brendle2013constant} becomes $\frac{2 f(r) f^{\prime}(r)}{r}-(n-2)\frac{\left(1-f(r)^{2}\right)}{r^{2}}$ is non-decreasing. We compute $\frac{2 f(r) f^{\prime}(r)}{r}-(n-2)\frac{\left(1-f(r)^{2}\right)}{r^{2}}=n$ which is non-decreasing.
By \cite[Lemma 3.8 and p. 261]{wang2017Minkowski}, $\overline g$ satisfies \nameref{assA}.

The condition (H4) in \cite{brendle2013constant} becomes $\frac{2 f(r) f^{\prime}(r)}{r}+\frac{1- f (r)^{2}}{ r^{2}}>0$. We compute $\frac{2 f(r) f^{\prime}(r)}{r}+\frac{1- f (r)^{2}}{ r^{2}}=1>0$.

This shows that \cite[Theorem 1.1]{brendle2013constant} also holds for a time-slice in the anti-deSitter spacetime.
\end{enumerate}
\end{example}

\begin{remark}
For a spacetime $V$ that satisfies \nameref{assA}, $Q=r d r \wedge d t$ is a conformal Killing-Yano two-form and $\mathrm{div}_{V} Q=\xi=-n \frac{\partial}{\partial t}$ is a Killing field.
\end{remark}

The concept of causal future or past is crucial in general relativity when studying the geometry of spacetimes. A particular case of interest is when a submanifold lies on the null hypersurface generated by a ``round sphere'', known as a ``shear-free'' null hypersurface. These shear-free null hypersurfaces are analogues of umbilical hypersurfaces in Riemannian geometry.

Here is the precise definition, following \cite{wang2017Minkowski}:
\begin{definition}\label{shear free}
An incoming null hypersurface $\underline{ C }$ is shear-free if there exists a spacelike hypersurface $\Sigma$ in $\underline{ C }$ such that the null second fundamental form $\underline{\chi}_{a b}=\left\langle D_{a} \underline{L}, e_{b}\right\rangle$ of $\Sigma$ with respect to some null normal $\underline{L}$ satisfies $\underline{\chi}_{a b}= \psi \sigma_{a b}$ for some function $\psi. $ A shear-free outgoing null hypersurface is defined in the same way.
\end{definition}

It was conjectured by Hopf \cite{hopf2003differential} that an oriented closed hypersurface with constant mean curvature immersed in the Euclidean space must be a sphere. Although later proved by Alexandrov \cite{alexsandorov1956uniqueness} to be true if the immersion assumption is replaced by embeddedness, the condition that the hypersurface is embedded is essential in proving Alexandrov-type theorems. Indeed, Wente \cite{wente1986counterexample} found a torus immersed in $\mathbb{R}^3$ with constant mean curvature, now known as the Wente's torus, which serves as a counterexample to Hopf's conjecture. Therefore, the condition that the hypersurface is embedded is essential in proving Alexandrov-type theorems. The following definition from \cite{wang2017Minkowski} can be regarded as the Lorentzian analogue of embeddedness:
\begin{definition}\label{def null embed}
A closed embedded spacelike codimension-two submanifold $\Sigma$ in a static spacetime $V$ is future (past, respectively) incoming null embedded if the future (past, respectively) incoming null hypersurface of $\Sigma$ intersects a totally geodesic time-slice $\{t=T\} \subset V$ at a smooth, embedded, orientable hypersurface $S$.
\end{definition}

We recall the following spacetime version of the Heintze-Karcher type inequality.

\begin{theorem}[{\cite[Theorem 3.12]{wang2017Minkowski}}]\label{HK ineq}
Let $V$ be a spacetime as in \nameref{assA}. Let $\Sigma \subset V$ be a future incoming null embedded closed spacelike codimension-two submanifold with $\langle\vec{H}, \underline{L}\rangle>0$, where $\underline{L}$ is a future incoming null normal. Then
\begin{equation} \label{3.14}
-(n-1) \int_{\Sigma} \frac{\left\langle\frac{\partial}{\partial t}, \underline{L}\right\rangle}{\langle\vec{H}, \underline{L}\rangle} d \mu-\frac{1}{2} \int_{\Sigma} Q(L, \underline{L}) d \mu \ge 0
\end{equation}
for a future outgoing null normal $L$ with $\langle L, \underline{L}\rangle=-2$ and $Q=r d r \wedge d t$ is the conformal Killing-Yano two-form on $V. $ Moreover, the equality holds if and only if $\Sigma$ lies in an incoming shear-free null hypersurface.
\end{theorem}

By reversing the time orientation, a similar inequality holds.

\begin{theorem}[{\cite[Theorem 3.12]{wang2017Minkowski}}]\label{HK ineq'}
Let $V$ be a spacetime as in \nameref{assA}. Let $\Sigma \subset V$ be a past incoming null embedded closed spacelike codimension-two submanifold with $\langle\vec{H}, {L}\rangle<0$, where ${L}$ is a future outgoing null normal. Then
\begin{equation} \label{3.17}
(n-1) \int_{\Sigma} \frac{\left\langle\frac{\partial}{\partial t}, {L}\right\rangle}{\langle\vec{H}, {L}\rangle} d \mu-\frac{1}{2} \int_{\Sigma} Q(L, \underline{L}) d \mu \ge 0
\end{equation}
for a future incoming null normal $\underline L$ with $\langle L, \underline{L}\rangle=-2$ and $Q=r d r \wedge d t$ is the conformal Killing-Yano two-form on $V. $ Moreover, the equality holds if and only if $\Sigma$ lies in an outgoing shear-free null hypersurface.
\end{theorem}

As in Theorem \ref{mink2}, we can rewrite Theorem \ref{HK ineq} and \ref{HK ineq'} as follows.
\begin{theorem}\label{HK ineq2}
With the same assumption as in Theorem \ref{HK ineq}, we have
\begin{align}\label{3.14'}
-(n-1) \int_{\Sigma} \frac{1}{|\vec{H}|^{2}} \left\langle\frac{\partial }{\partial t }, \vec{H}+\vec{J}\right\rangle d \mu+ \int_{\Sigma} \frac{1}{|\vec{H}|^{2}} Q(\vec{H}, \vec{J}) d \mu \ge 0.
\end{align}
Moreover, the equality holds if and only if $\Sigma$ lies in an incoming shear-free null hypersurface.
\end{theorem}
\begin{theorem}
With the same assumption as in Theorem \ref{HK ineq'}, we have
\begin{equation}\label{3.14''}
-(n-1) \int_{\Sigma} \frac{1}{|\vec{H}|^{2}}\left\langle\frac{\partial}{\partial t}, -\vec{H}+\vec{J}\right\rangle d \mu+\int_{\Sigma} \frac{1}{|\vec{H}|^{2}} Q(\vec{H}, \vec{J}) d \mu \ge 0.
\end{equation}
The equality holds if and only if $\Sigma$ lies in an outgoing shear-free null hypersurface.
\end{theorem}

\section{An integral Alexandrov type theorem}\label{sec integral Alexandrov type theorem}
The following proposition explains the appearance of the quantity $Q(\nabla(\log |\vec{H}|)-\sum_{a=1}^{n-1} \alpha_{\vec{H}}\left(e_a\right) e_a, \vec{H}+\vec{J})$ in our Theorem \ref{thm intro1}.
\begin{proposition}
Suppose the mean curvature vector $\vec H$ of $\Sigma$ is spacelike. Define the two null vectors
$\underline{L}=\frac{1}{|\vec{H}|}\left(\frac{\vec{H}}{|\vec{H}|}+\frac{\vec{J}}{|\vec{H}|}\right)$ and $L=|\vec{H}|\left(-\frac{\vec{H}}{|\vec{H}|}+\frac{\vec{J}}{|\vec{H}|}\right)$. Then $\langle L, \underline{L}\rangle=-2$, $\langle H, \underline L\rangle =1$ and
\begin{equation*}
d \log |\vec{H}|-\alpha_{\vec{H}}= \frac{1}{2}\left\langle D \underline{L}, L\right\rangle.
\end{equation*}
\end{proposition}

\begin{proof}
It is easy to verify that $\langle L, \underline{L}\rangle=-2$ and $\langle \vec H, \underline L\rangle =1$.
If follows from the computation in \cite[Proposition 3.4]{wang2017Minkowski} that
$\frac{1}{2}\left\langle D_{e_a} \underline L, L \right\rangle
=d\log |\vec{H}|(e_a) -\alpha_{\vec H} (e_a)$ and the result follows.
\end{proof}

Now, we are ready to state our first main result, which generalizes \cite[Theorem 3.14]{wang2017Minkowski}.

\begin{theorem}\label{alex}
Let $V$ be a spherically symmetric spacetime as in \nameref{assA} and $\Sigma$ be a
closed embedded, spacelike codimension-two submanifold in $V$ such that $\vec H$ is spacelike.

\begin{enumerate}
\item Assume $\Sigma$ is future incoming null embedded. Then $\Sigma$ satisfies
\begin{equation}\label{cond}
\int_\Sigma \frac{1}{|\vec{H}|^{2}} Q\left(\nabla(\log |\vec{H}|)-\sum_{a=1}^{n-1} \alpha_{\vec{H}}\left(e_{a}\right) e_{a}, \vec{H}+\vec{J}\right)d\mu\le 0
\end{equation}
if and only if $\Sigma$ lies in an incoming shear-free null hypersurface.
\item Assume $\Sigma$ is future outgoing null embedded. Then $\Sigma$ satisfies
\begin{equation*}
\int_\Sigma \frac{1}{|\vec{H}|^{2}} Q\left(\nabla(\log |\vec{H}|)+\sum_{a=1}^{n-1} \alpha_{\vec{H}}\left(e_{a}\right) e_{a}, -\vec{H}+\vec{J}\right)d\mu\le 0
\end{equation*}
if and only if $\Sigma$ lies in an outgoing shear-free null hypersurface.
\end{enumerate}
\end{theorem}

\begin{proof}
Let $f=\frac{1}{|\vec H| }$. Assume $\Sigma$ is future incoming null embedded.
From the assumption \eqref{cond}, the weighted spacetime Minkowski formula \eqref{2.2'} becomes
\begin{align*}
& -(n-1) \int_{\Sigma} \frac{1}{|\vec{H}|^{2}}\left\langle\frac{\partial}{\partial t}, \vec{H}+\vec{J}\right\rangle d \mu+\int_{\Sigma} \frac{1}{|\vec{H}|^{2}} Q(\vec{H}, \vec{J}) d \mu\\
=& \int_\Sigma \frac{1}{|\vec{H}|} Q\left(\frac{\nabla|\vec H|}{|\vec H|^{2}} - \sum_{a=1}^{n-1}\frac{\alpha_{\vec H}\left(e_{a}\right) e_{a}}{|\vec H|}, \vec{H}+\vec{J}\right)d\mu\\
=& \int_\Sigma\frac{1}{|\vec{H}|^2} Q\left(\nabla\left(\log|\vec H|\right) - \sum_{a=1}^{n-1}\alpha_{\vec H}\left(e_{a}\right) e_{a}, \vec{H}+\vec{J}\right)d\mu\\
\le& 0.
\end{align*}
Hence the equality is achieved in the spacetime Heintze-Karcher inequality \eqref{3.14'} and we conclude that $\Sigma$ lies in an incoming shear-free null hypersurface.

Moreover, if \eqref{cond} does not hold, i.e., $\int_{\Sigma} \frac{1}{|\vec{H}|^2} Q\left(\nabla(\log |\vec{H}|)-\sum_{a=1}^{n-1} \alpha_{\vec{H}}\left(e_a\right) e_a, \vec{H}+\vec{J}\right) d \mu>0$, then the above argument shows that
$$
-(n-1) \int_{\Sigma} \frac{1}{|\vec{H}|^2}\left\langle\frac{\partial}{\partial t}, \vec{H}+\vec{J}\right\rangle d \mu+\int_{\Sigma} \frac{1}{|\vec{H}|^2} Q(\vec{H}, \vec{J}) d \mu>0.
$$
Therefore, the spacetime Heintze-Karcher inequality \eqref{3.14'} is a strict inequality, and hence $\Sigma$ does not lie in an incoming shear-free null hypersurface.

The proof of the future outgoing null embedded case is the same by applying \eqref{2.2''} and \eqref{3.14''} instead.
\end{proof}

If $N=(\mathbb S^{n-1}, g_{\mathbb S^{n-1}})$ in \eqref{3.10}, the spacetime is spherically symmetric. A null hypersurface in an $(n+1)$-dimensional spherically symmetric spacetime is called a null hypersurface of symmetry if it is invariant under the $SO(n)$ isometry.
In other words, it is generated by a sphere of symmetry.

The exterior Schwarzschild spacetime metric \eqref{1.2}
and the anti-deSitter spacetime metric \eqref{ads}
both satisfy \nameref{assA}. Since the spheres of symmetry are the only closed umbilical hypersurfaces in the totally geodesic time slice of
these spacetimes \cite[Theorem 1.1, Corollary 1.2]{brendle2013constant} (see also the remark in Example \ref{example1} (2)), as a direct corollary of Theorem \ref{alex}, we obtain:
\begin{theorem}\label{thm Schwarzschild}
With the same assumptions in Theorem \ref{alex}, and assume further that
$V$ is either the Schwarzschild spacetime or the anti-deSitter spacetime,
then the ``shear-free null hypersurface'' in the conclusion of Theorem \ref{alex} can be replaced by ``null hypersurface of symmetry''.
\end{theorem}

The following result generalizes the fact that a closed embedded spacelike
codimension-two submanifold with parallel mean curvature vector in the
Schwarzschild spacetime is a sphere of symmetry, see \cite[Cor. C]{wang2017Minkowski} and also \cite{chen1973surface}, \cite{yau1974submanifolds}.
\begin{corollary}
Let $\Sigma$ be a closed embedded spacelike codimension-two submanifold with space like mean curvature $\vec H$ in the Schwarzschild spacetime or the anti-deSitter spacetime. Suppose $\Sigma$ is both future and past incoming null embedded.
Then a necessary and sufficient condition for $\Sigma$ to be a sphere of symmetry is that
\begin{align*}
\int_{\Sigma} \frac{1}{|\vec{H}|^2} Q\left(\nabla(\log |\vec{H}|)-\sum_{a=1}^{n-1} \alpha_{\vec{H}}\left(e_a\right) e_a, \vec{H}+\vec{J}\right) d \mu \le 0
\end{align*}
and
\begin{align*}
\int_{\Sigma} \frac{1}{|\vec{H}|^2} Q\left(\nabla(\log |\vec{H}|)+\sum_{a=1}^{n-1} \alpha_{\vec{H}}\left(e_a\right) e_a, -\vec{H}+\vec{J}\right) d \mu \le 0.
\end{align*}
\end{corollary}
\begin{proof}
Theorem \ref{alex} and \ref{thm Schwarzschild} imply that $\Sigma$ is the intersection of one incoming and one outgoing null hypersurface of symmetry. Therefore, $\Sigma$ is a sphere of symmetry.
\end{proof}
\section{Generalized Alexandrov theorems for mixed higher order mean curvatures in Lorentzian space forms}\label{sec higher}

In this section, we will first introduce the mixed higher order mean curvatures mixed higher order mean curvature $P_{r, s} $ of a codimension-two spacelike submanifold in an $(n+1)$-dimensional spacetime $V$, and the corresponding weighted Minkowski formulas on $\Sigma$. Since the divergence of the corresponding Newton tensor now involves the ambient curvature, the Minkowski formulas are more complicated than the one in Section \ref{sec mink} in general.

In \cite[Section 4, Section 5]{wang2017Minkowski}, Minkowski formulas and Alexandrov type theorems were established under the assumption that the spacetime has a constant curvature and the submanifold $\Sigma$ is torsion-free. We will demonstrate that these Alexandrov type theorems can be generalized in three aspects: Firstly, the pointwise condition can be substituted with a condition on an integral. Secondly, the constancy condition can be replaced with an inequality condition. Thirdly, we can replace the torsion-free condition, which depends on the choice of null frame $(L, \underline L)$, with a cohomology condition that is independent of the choice of null frame.

The torsion-free condition is that the connection $1$-form $\zeta_L=0$, defined below. However, it is not always clear whether there exists a null frame in which this condition is true, since it is not a condition that is invariant under the change of null frame $L\to fL, \underline L\to \frac{1}{f}\underline L$.
The last improvement is possible thanks to the observation that while the connection $1$-form $\zeta_L$ depends on the choice of the null frame, its exterior derivative does not. This allows us to relax the torsion-free condition to the condition that $d \zeta=0$ and $[\zeta]=0 \in H_{d R}^1(\Sigma)$, which we consider more natural.

Let us define the necessary terms.
Let $L$ and $\underline{L}$ be two null normals of $\Sigma$ such that $L$ is future outgoing and $\langle L, \underline{L}\rangle=-2$. The null second fundamental forms with respect to $L, \underline{L}$ are defined by
\begin{align*}
\chi_{a b}=\left\langle D_{e_{a}} L, e_{b}\right\rangle, \quad
\underline{\chi}_{a b}=\left\langle D_{e_{a}} \underline{L}, e_{b}\right\rangle,
\end{align*}
and we write $\zeta=\zeta_{L}$ for the connection $1$-form with respect to $L$:
\begin{align*}
\zeta_{a}=\frac{1}{2}\left\langle D_{e_{a}} L, \underline{L}\right\rangle.
\end{align*}

We say that $\Sigma$ is torsion-free with respect to the null frame $L$ and $\underline{L}$ if $\zeta_L=0$. Notice that this condition depends on the choice of $L$ and $\underline L$.

In codimension-two spacelike submanifolds, the definition of higher order mean curvatures becomes more complex compared to the hypersurface case in the Riemannian setting \cite{reilly1973variational}. This is because there are two null normal directions, and we must combine both of them in the definition, resulting in ``mixed'' Newton tensors and higher order mean curvatures.
\begin{definition}
For any two non-negative integers $r$ and $s$ with $0 \le $ $r+s \le n-1$, the mixed higher order mean curvature $P_{r, s}(\chi, \chi)$ with respect to $L$ and $\underline{L}$ is defined through the following expansion:
$$
\det (\sigma+y \chi+\underline{y}\, \underline {\chi})=\sum_{0 \le r+s \le n-1} \frac{(r+s) !}{r ! s !} y^{r} \underline{y}^{s} P_{r, s}(\chi, \underline{\chi})
$$
where $y$ and $\underline y$ are two real variables and $\sigma$ is the induced metric on $\Sigma$.
Sometimes, it is more convenient to use the normalized mixed higher order mean curvature, defined by
$$
H_{r, s}(\chi, \underline \chi)=\frac{1}{\binom{n-1}{r+s}} P_{r, s}(\chi, \underline \chi).
$$

We also define the mixed Newton tensors $T_{r, s}(\chi, \underline{\chi})$ and $\underline{T}_{r, s}(\chi, \underline{\chi})$ on $\Sigma$, which are symmetric $2$-tensors, by
$$
T_{r, s}^{a b}(\chi, \underline{\chi})=\frac{\partial }{\partial \chi_{a b}}P_{r, s}(\chi, \underline{\chi}) \quad \textrm { and } \quad \underline{T}_{r, s}^{a b}(\chi, \underline{\chi})=\frac{\partial }{\partial \underline{\chi}_{a b}}P_{r, s}(\chi, \underline{\chi}).
$$
\end{definition}

If the choice of $L$ and $\underline L$ is clear, we will write $P_{r, s}$ for $P_{r, s}(\chi, \underline{\chi}), T_{r, s}^{a b}$ for $T_{r, s}^{a b}(\chi, \chi)$, etc.
\begin{lemma} \label{lem: d zeta}
The two-form $d\zeta_L$ is independent of the choice of $L$ and $\underline L$.
\end{lemma}
\begin{proof}
The Ricci identity states that
$$
(d \zeta_L)_{a b}=(\nabla_{e_a}\zeta_L)(e_b)-(\nabla_{e_b}\zeta_L)(e_a)=\left\langle\overline {R}\left(e_{a}, e_{b}\right) L, \underline{L}\right\rangle + \chi_{b}^{c} \underline{\chi}_{c a} - \chi_{a}^{c} \underline{\chi}_{c b}.
$$
As the RHS is invariant under $L\to fL$ and $\underline L\to \frac{1}{f}\underline L$, the result follows.
\end{proof}

By Lemma \ref{lem: d zeta}, it makes sense to talk about $d\zeta$ without specifying the choice of $L$ and $\underline L$. Moreover,
under the transformation $L \rightarrow f L=\widetilde{L}$ and $\underline{L} \rightarrow \frac{1}{f} \underline{L}=\widetilde{\underline L}$, a direct calculation yields $\widetilde{\zeta}_{b}=\zeta_{b}-(\log f)_{b}$. In other words, if $d\zeta=0$, then $[\zeta]$ is well-defined in the deRham cohomology space $H^1_{dR}(\Sigma)$.

\begin{proposition}\label{prop torsion}
Suppose
$d\zeta=0$ and
$[\zeta]=0$ in $H^1_{dR}(\Sigma)$, then we can find two null normals $ L$ and $ {\underline L}$ with $\langle L, \underline L\rangle =-2$, such that $\zeta_{ L}=0$. The choice of $L$ is unique up to a positive constant.
\end{proposition}

\begin{proof}
For any null normal $L$, we have $ \zeta_{fL}(X)=\zeta_{L}(X)-(d\log f)(X)$.
In order for $ \zeta_{fL}=0$, we need to solve $d(\log f)=\zeta_L$.
As $[\zeta_L]=0\in H^1_{dR}(\Sigma)$, there exists a smooth function $h$ such that $dh=\zeta_L$. Therefore we can replace $L$ by $fL $ and $\underline L$ by $\frac{1}{f}\underline L$, where $f=e^h$. As the function $h$ is unique up to an additive constant, $f$ and hence such $L$ is unique up to a multiplicative constant.
\end{proof}

\begin{corollary}
Assume $\Sigma$ has cohomology $H^1 (\Sigma, \mathbb R) =0$ and $d\zeta=0$, then we can find two null normals $ L$ and $ {\underline L}$ with $\langle L, \underline L\rangle =-2$, such that $\zeta_{ L}=0$. The choice of $L$ is unique up to a positive constant.
\end{corollary}

In particular, if $\Sigma$ is a topological sphere and $n\ge 3$, then $d\zeta=0$ implies the torsion-free condition for some choice of $L$ and $\underline L$.

In the rest of this section, we focus on spacelike codimension-two submanifolds in a spacetime of constant curvature such as the Minkowski spacetime $\mathbb R^{n, 1}$, the de-Sitter spacetime, or the anti de-Sitter spacetime.

The following is the analogue of the fact that for a hypersurface in a Riemannian space form, the Newton tensors are divergence-free \cite[Lemma 2.1]{kwong2015inequality}. In the codimension-two case, an additional torsion-free condition (which is automatic for a hypersurface) is required.
\begin{lemma}\label{lem L}
Let $ \Sigma$ be a spacelike codimension-two submanifold in a spacetime of constant curvature such that $ d\zeta=0$ and $[\zeta]=0\in H^1_{dR}(\Sigma)$. Then there exists two null normals $L$ and $\underline{L} $ with $\langle L, \underline L\rangle =-2$ such that $T_{r, s}(\chi, \underline \chi)$ and $\underline{T}_{r, s}(\chi, \underline \chi)$ are divergence free for any $(r, s)$, that is
$$
\nabla_{b} T_{r, s}^{a b}=\nabla_{b} \underline{T}_{r, s}^{a b}=0.
$$
\end{lemma}
\begin{proof}
By Proposition \ref{prop torsion}, we can find such $L$ and $\underline L$ so that $\zeta_L=0$. The result then follows from \cite[Lemma 4.2]{wang2017Minkowski}.
\end{proof}

We extend \cite[Theorem 4.3]{wang2017Minkowski} by adding an arbitrary weight to the curvature integrals.
\begin{theorem}\label{thm higher mink}
Let $ \Sigma$ be a spacelike codimension-two submanifold in an $(n+1)$-dimensional spacetime of constant curvature such that $ d\zeta=0$ and $[\zeta]=0\in H^1_{dR}(\Sigma)$. Then for $L$ and $\underline L$ given by Proposition \ref{prop torsion}, and for any smooth function $f$ on $\Sigma$, we have
\begin{align}\label{higher weight}
\frac{r(n-r-s)}{r+s} \int_{\Sigma}f P_{r-1, s}\left\langle L, \frac{\partial }{\partial t} \right\rangle d\mu+\frac{1}{2} r \int_{\Sigma} fP_{r, s} Q(L, \underline L)d\mu+\int_{\Sigma} Q(L, T_{r, s}(\nabla f))d\mu=0
\end{align}
and
\begin{align}\label{higher weight2}
\frac{s(n-r-s)}{r+s} \int_{\Sigma}f P_{r, s-1}\left\langle\underline L, \frac{\partial }{\partial t} \right\rangle d\mu-\frac{1}{2}s \int_{\Sigma} fP_{r, s} Q(L, \underline L)d\mu+\int_{\Sigma} Q(\underline L, \underline T_{r, s}(\nabla f))d\mu=0.
\end{align}
\end{theorem}

\begin{proof}
Write $T_{r, s}$ as $T$.
By the torsion-free condition from Proposition \ref{prop torsion}, constant ambient curvature condition and the divergence-free property of $T$ by Lemma \ref{lem L}, we have \cite[p.267]{wang2017Minkowski}
$$
\nabla_{a}\left[T^{a b} Q\left(L, e_b\right)\right]= T^{a b} \sigma_{a b}\left\langle L, \frac{\partial}{\partial t}\right\rangle
+\frac{1}{2}\left(T^{a b} \chi_{a b}\right) Q(L, \underline{L}).
$$
Therefore
\begin{equation}\label{div}
\nabla_{a}\left[fT^{a b} Q\left(L, e_b\right)\right]
=f T^{a b} \sigma_{a b}\left\langle L, \frac{\partial}{\partial t}\right\rangle
+\frac{1}{2}f\left(T^{a b} \chi_{a b}\right) Q(L, \underline{L})+Q(L, T(\nabla f)).
\end{equation}

Integrating \eqref{div} on $\Sigma$ and applying the divergence theorem together with the fact that $ \chi_{a b} T_{r, s}^{a b}=r P_{r, s}$ and $ \sigma_{a b} T_{r, s}^{a b}=\frac{r(n-(r+s))}{r+s} P_{r-1, s}$ (\cite[A.3, A.5]{wang2017Minkowski}),
we can get the result.

The second formula is derived similarly by considering the identity
$\int_{\Sigma} \nabla_{a}\left[\underline{T}^{a b} Q\left(\underline{L}, e_b\right)\right] d \mu=0$ and using $\nabla_a\left[Q\left(\underline{L}, e_b\right)\right]=\left(D_a Q\right)\left(\underline{L}, e_b\right)+\underline{\chi}_a^c Q_{c b}-\frac{1}{2} \underline{\chi}_{a b} Q(L, \underline{L})$.
\end{proof}

Recall the definition of the positive cone $\Gamma_{k}$. For $1 \le k \le n-1, \Gamma_{k}$ is a convex cone in $\mathbb R^{n-1}$ such that $\Gamma_{k}=\{\lambda \in \mathbb R^{n-1}: \sigma_{1}(\lambda)>0, \cdots, \sigma_{k}(\lambda)>0\}$ where
$$
\sigma_{k}(\lambda)=\sum_{i_{1}<\cdots<i_{k}} \lambda_{i_{1}} \cdots \lambda_{i_{k}}
$$
is the $k$-th elementary symmetric function.

We have the following Newton-Maclaurin inequality, see \cite{wang2017Minkowski} A.9.
\begin{lemma}\label{lem newton}
If $\chi$ and $\underline{\chi}$ are both in the $\Gamma_{r+s-1}$ cone, then we have
\begin{equation}\label{newton ineq}
H_{r-1, s}(\chi, \underline \chi)^2 \ge H_{r, s}(\chi, \underline \chi) H_{r-2, s}(\chi, \underline \chi).
\end{equation}
\end{lemma}

The following result generalizes \cite[Theorem 5.1]{wang2017Minkowski}.
\begin{theorem}\label{thm higher order alex}
Let $\Sigma$ be a past incoming null embedded, closed spacelike codimension-two submanifold in an $(n+1)$ dimensional spacetime of constant curvature
such that $d \zeta=0$ and $[\zeta]=0 \in H_{d R}^1(\Sigma)$. Let $L$ and $\underline{L}$ be the null frame given by Proposition \ref{prop torsion}.
Assume that the second fundamental form $\chi \in \Gamma_{r}$.
If
\begin{align*}
\int_{\Sigma} \frac{1}{{P_{r, 0}}^2}Q\left(T_{r, 0}(\nabla P_{r, 0}), L\right) d \mu\le 0,
\end{align*}
then $\Sigma$ lies in an outgoing shear-free null hypersurface.
\end{theorem}

\begin{proof}
We can rewrite \eqref{higher weight} as
\begin{equation}\label{higher weight'}
\int_\Sigma f H_{r-1, s}\left\langle L, \frac{\partial}{\partial t}\right\rangle d\mu+\frac{1}{2} \int_\Sigma f H_{r, s} Q(L, \underline L)d\mu
-\frac{r+s}{\binom{n-1}{r-1+s} r(n-r-s)}\int_{\Sigma} Q\left(T_{r, s}(\nabla f), L \right) d \mu=0.
\end{equation}
We apply \eqref{higher weight'} with $f=\frac{1}{H_{r, 0}}$ and $s=0$ to get
\begin{align*}
\frac{1}{2} \int_{\Sigma} Q(L, \underline{L}) d \mu
=& -\int \frac{H_{r-1, 0}}{H_{r, 0}}\left\langle L, \frac{\partial}{\partial t}\right\rangle d\mu-c(n, r) \int_{\Sigma} \frac{1}{H_{r, 0}^2} Q\left(T_{r, 0}\left(\nabla H_{r, 0}\right), L\right) d \mu\\
\ge& -\int \frac{H_{r-1, 0}}{H_{r, 0}}\left\langle L, \frac{\partial}{\partial t}\right\rangle d\mu,
\end{align*}
where $c(n, r)>0$.
Note that $\left\langle L, \frac{\partial }{\partial t}\right\rangle <0$ and $P_{1, 0}=\mathrm{tr}(\chi)=\sigma^{ab}\langle -D_a e_b, L\rangle =-\langle \vec H, L\rangle $. By repeatedly applying the Newton-Maclaurin inequality \eqref{newton ineq}, we then get
\begin{align*}
\frac{1}{2} \int_{\Sigma} Q(L, \underline{L}) d \mu\ge& -\int_{\Sigma} \frac{1}{H_{1, 0}}\left\langle L, \frac{\partial}{\partial t} \right\rangle d\mu
= (n-1) \int_{\Sigma} \frac{\left\langle L, \frac{\partial}{\partial t} \right\rangle}{\langle\vec{H}, L\rangle} d \mu.
\end{align*}
Comparing this with the spacetime Heintze-Karcher inequality \eqref{3.17}, we see that the equality is achieved.
We conclude that $\Sigma$ lies in an outgoing shear-free null hypersurface.
\end{proof}

\begin{theorem}\label{thm higher order alex'}
Let $\Sigma$ be a future incoming null embedded, closed spacelike codimension-two submanifold in an $(n+1)$ dimensional spacetime of constant curvature
such that $d \zeta=0$ and $[\zeta]=0 \in H_{d R}^1(\Sigma)$. Let $L$ and $\underline{L}$ be the null frame given by Proposition \ref{prop torsion}.
Assume that the second fundamental form $-\underline{\chi} \in \Gamma_{s}$. If $$\int_{\Sigma}\frac{1}{ {P_{0, s}}^{2}} Q\left(T_{0, s}(\nabla P_{0, s}), \underline L\right) d \mu\le 0, $$
then $\Sigma$ lies in an incoming shear-free null hypersurface.
\end{theorem}

\begin{proof}
The proof is similar to Theorem \ref{thm higher order alex}.
Using \eqref{higher weight2} with $f=\frac{1}{H_{0, s}(\chi, \underline \chi)}$ and $r=0$, we get
\begin{align*}
\frac{1}{2} \int_{\Sigma} Q(L, \underline{L}) d \mu
=& \int_\Sigma\frac{H_{0, s-1 }(\chi, \underline \chi)}{H_{ 0, s}(\chi, \underline \chi)}\left\langle \underline L, \frac{\partial}{\partial t}\right\rangle d\mu+c(n, s)\int_{\Sigma}\frac{1}{H_{0, s}(\chi, \underline \chi)^2}Q(T_{r, s}(\nabla H_{0, s}(\chi, \underline \chi)), \underline L)d\mu\\
\ge& \int_{\Sigma} \frac{H_{0, s-1}(\chi, \underline{\chi})}{H_{0, s}(\chi, \underline{\chi})}\left\langle\underline{L}, \frac{\partial}{\partial t}\right\rangle d \mu\\
=& -\int_{\Sigma} \frac{H_{0, s-1}(\chi, -\underline{\chi})}{H_{0, s}(\chi, -\underline{\chi})}\left\langle\underline{L}, \frac{\partial}{\partial t}\right\rangle d \mu,
\end{align*}
where $c(n, s)>0$.
Note that $\left\langle \underline L, \frac{\partial }{\partial t}\right\rangle <0$ and $P_{1, 0}(\chi, -\underline \chi)=\mathrm{tr}(-\underline \chi)=\sigma^{ab}\langle D_a e_b, \underline L\rangle =\langle \vec H, \underline L\rangle $. Again by repeatedly applying the Newton-Maclaurin inequality, we get
\begin{align*}
\frac{1}{2} \int_{\Sigma} Q(L, \underline{L}) d \mu\ge -\int_{\Sigma} \frac{1}{H_{0, 1}(\chi, -\underline \chi)}\left\langle \underline L, \frac{\partial}{\partial t} \right\rangle d\mu
= -(n-1) \int_{\Sigma} \frac{\left\langle \underline L, \frac{\partial}{\partial t} \right\rangle}{\langle\vec{H}, \underline L\rangle} d \mu.
\end{align*}
Comparing this with the spacetime Heintze-Karcher inequality \eqref{3.14}, we see that the equality is achieved.
We conclude that $\Sigma$ lies in an incoming shear-free null hypersurface.
\end{proof}

In the rest of this section, we prove a rigidity result for submanifolds with a condition that involves $P_{r, s}$ for $r>0, s>0$.

We have the following algebraic lemma from \cite{wang2017Minkowski}.
\begin{lemma}[{\cite[Lemma 5.2]{wang2017Minkowski}}]\label{5.2}
Suppose $\chi \in \Gamma_{r+s}$ and $-\underline{\chi} \in \Gamma_{r+s}$. If
$\underline{\chi}_{a b}\left(\underline{T}_{0, s}\right)^{b c} \chi_c{ }^a \ge H_{0, s} H_{1, 0}$,
then we have
$$
\frac{H_{r-1, s}(\chi, \underline{\chi})}{H_{r, s}(\chi, \underline{\chi})} \ge \frac{n-1}{\mathrm{tr} \chi}.
$$
The equality holds if and only if $\chi$ is a multiple of the identity.
\end{lemma}
\begin{theorem}
Let $\Sigma$ be a past incoming null embedded closed embedded spacelike codimension-two submanifold in a spacetime $V^{n+1}$ of constant curvature,
such that $d \zeta=0$ and $[\zeta]=0 \in H_{d R}^1(\Sigma)$. Let $L$ and $\underline{L}$ be the null frame given by Proposition \ref{prop torsion}.
Let $r, s>0$ and suppose the second fundamental forms $\chi \in \Gamma_{r+s}$ and $-\underline{\chi} \in \Gamma_{r+s}$ satisfy $\underline{\chi}_{a b}\left(\underline{T}_{0, s}\right)^{b c} \chi_c{ }^a \ge H_{0, s} H_{1, 0}$ and
\begin{equation}\label{cond rs}
\int_{\Sigma} \frac{1}{H_{r, s}^2} Q\left(T_{r, s}\left(\nabla H_{r, s}\right), L\right) d \mu\le0.
\end{equation}
Then $\Sigma$ is a sphere of symmetry.
\end{theorem}
\begin{proof}
Putting $f=\frac{1}{H_{r, s}}$ in \eqref{higher weight'}, we get
\begin{equation*}
\begin{split}
-\int_{\Sigma} \frac{ H_{r-1, s}}{H_{r, s}}\left\langle L, \frac{\partial}{\partial t}\right\rangle d \mu=& \frac{1}{2} \int_{\Sigma} Q(L, \underline{L}) d \mu+c(n, r, s)\int_{\Sigma}\frac{1}{H_{r, s}^2} Q\left(T_{r, s}(\nabla H_{r, s}), L\right) d \mu\\
\le& \frac{1}{2} \int_{\Sigma} Q(L, \underline{L}) d \mu,
\end{split}
\end{equation*}
where $c(n, r, s)>0$.
It follows from Lemma \ref{5.2} and the fact that $-\left\langle L, \frac{\partial}{\partial t}\right\rangle \ge 0$,
$$
(n-1) \int_{\Sigma} \frac{\left\langle L, \frac{\partial}{\partial t}\right\rangle}{\langle\vec{H}, L\rangle} d \mu=-(n-1) \int_{\Sigma} \frac{\left\langle L, \frac{\partial}{\partial t}\right\rangle}{\mathrm{tr} \chi} d \mu \le \frac{1}{2} \int_{\Sigma} Q(L, \underline{L}) d \mu.
$$
Comparing this with the spacetime Heintze-Karcher inequality \eqref{3.17}, we see that the equality is achieved. Therefore by the equality case in Lemma \ref{5.2}, we have $\chi=\alpha\sigma$ for some positive smooth function $\alpha$. Moreover, $\Sigma$ lies in an outgoing shear-free null hypersurface.

The Codazzi equation gives $\nabla_a \chi_{b d}-\nabla_b \chi_{a d}=\left\langle\overline {R}\left(e_a, e_b\right) L, e_d\right\rangle+\zeta_b \chi_{a d}-\zeta_a \chi_{b d}=0$ since the ambient curvature is constant and the torsion $\zeta_L=0$. From this, we have $0= \nabla_a \chi_{b a}-\nabla_b \chi_{a a}= \nabla_b \alpha - (n-1)\nabla_b \alpha $ and hence $\alpha$ is constant (note that $r, s>0$ implies $n\ge 3$), cf. also \cite[Theorem 2.2]{kwong2015inequality}.

It is then not hard to see that $P_{r, s}(\chi, \underline \chi)=\alpha^r P_{ 0, s }(\chi, \underline \chi)$ and $T_{r, s}=\alpha^{r-1}T_{0, s}$. So the condition \eqref{cond rs} becomes
$$
\int_{\Sigma} \frac{1}{H_{0, s}^2} Q\left(T_{0, s}\left(\nabla H_{0, s}\right), L\right) d \mu \le 0.
$$
This falls in the setting of Theorem \ref{thm higher order alex'}. Hence $\Sigma$ lies in an incoming shear-free null hypersurface.
As $\Sigma$ is the intersection of one incoming and one outgoing null hypersurface of symmetry, $\Sigma$ is a sphere of symmetry.
\end{proof}

\section{Integral Alexandrov theorems in the Schwarzschild spacetime}\label{sec Schwarzschild}

In this section, we prove some generalized Alexandrov theorems in the Schwarzschild spacetime. In this case, the ambient curvature is not constant and the mixed Newton tensors are no longer divergence free. Nevertheless, the divergence term $\nabla_aT^{ab}_{r, 0}Q(L, e_b)$ or
$\nabla_aT^{ab}_{0, s}Q(\underline L, e_b)$ still has a favourable sign under some natural assumptions on $\Sigma$, which enable us to derive a weighted Minkowski-type inequality. This is similar to the observation made in \cite[Proposition 8]{BE2013} in order to obtain a Minkowski-type inequality in the Riemannian Schwarzschild manifold.
\begin{lemma}[{\cite[Lemma 6.1]{wang2017Minkowski}}]\label{lem 6.1}
Let $\Sigma$ be a spacelike codimension-two submanifold in the Schwarzschild spacetime such that $d \zeta=0$ and $[\zeta]=0 \in H_{d R}^1(\Sigma)$.
Let $L$ and $\underline{L}$ be the null frame given by Proposition \ref{prop torsion} and $\left(Q^2\right)_{\alpha \beta}:=Q_\alpha^\gamma Q_{\gamma \beta}$.
Then the following statements are true:
\begin{enumerate}
\item\label{(1)}
If $Q(L, \underline{L}) \ge 0$, then $\left(\nabla_a T_{2, 0}^{ab}\right) Q\left(L, e_b\right) \le 0$ and $\left(\nabla_a T_{0, 2}^{ab}\right) Q\left(\underline{L}, e_b\right) \le 0$.
\item\label{(2)}
Suppose $\chi>0$ and $ Q^2 (L, v) Q(L, v) \le 0$ for any $v\in T\Sigma$, then $\left(\nabla_a T_{r, 0}^{ab}\right) Q\left(L, e_b\right) \le 0$ if $r \ge 3$.
\item\label{(3)}
Suppose $-\underline{\chi}>0$ and $ Q^2 (\underline{L}, v) Q(\underline{L}, v) \ge 0$ for any $v\in T\Sigma$, then $\left(\nabla_a T_{0, s}^{ab}\right) Q\left(\underline{L}, e_b\right) \le 0$ if $s \ge 3$.
\end{enumerate}
\end{lemma}

In the Schwarzschild case, we can obtain weighted higher order Minkowski inequalities instead of exact formulas.
\begin{theorem}\label{thm schwarz mink}
Let $\Sigma$ be a closed spacelike codimension-two submanifold in the Schwarzschild spacetime such that $d \zeta=0$ and $[\zeta]=0 \in H_{d R}^1(\Sigma)$. Let $L$ and $\underline{L}$ be the null frame given by Proposition \ref{prop torsion} and $f$ be a smooth positive function on $\Sigma$.
\begin{enumerate}
\item If $\Sigma$ satisfies assumption \eqref{(1)} or \eqref{(2)} in Lemma \ref{lem 6.1}, then for any $1 \le $ $r \le n-1$,
\begin{equation*}\begin{split}
\int_{\Sigma} f H_{r-1, 0}\left\langle L, \frac{\partial}{\partial t}\right\rangle d \mu+\frac{1}{2} \int_{\Sigma} f H_{r, 0} Q(L, \underline{L}) d \mu-\frac{1}{\binom{n-1}{r-1}(n-r)} \int_{\Sigma} Q\left(T_{r, 0}(\nabla f), L\right) d \mu \ge 0.
\end{split}\end{equation*}
\item If $\Sigma$ satisfies assumption \eqref{(1)} or \eqref{(3)} in Lemma \ref{lem 6.1}, then for any $1 \le $ $r \le n-1$,
\begin{equation*}\begin{split}
\int_{\Sigma} f H_{0, s-1}\left\langle\underline{L}, \frac{\partial}{\partial t}\right\rangle d \mu-\frac{1}{2} \int_{\Sigma} f H_{0, s} Q(L, \underline{L}) d \mu-\frac{1}{\binom{n-1}{s-1}(n-s)} \int_{\Sigma} Q\left(T_{0, s}(\nabla f), \underline{L}\right) d \mu \ge 0.
\end{split}\end{equation*}
\end{enumerate}
\end{theorem}
\begin{proof}
Note that $T_{1, 0}=T_{0, 1}={I}$ and so are divergence-free.
Let $T=T_{r, 0}$.

If $\Sigma$ satisfies assumption \eqref{(1)} or \eqref{(2)} in Lemma \ref{lem 6.1}, then a similar calculation as in Theorem \ref{thm higher mink} gives
\begin{equation*}
\begin{aligned}
\nabla_a\left[f T^{a b} Q\left(L, e_b\right)\right]
=& f\nabla_aT^{ab} Q(L, e_b)+fT^{a b} \sigma_{a b}\left\langle L, \frac{\partial}{\partial t}\right\rangle \\
& + \frac{1}{2}f\left(T^{a b} \chi_{a b}\right) Q(L, \underline{L}) +Q(T(L, \nabla f))\\
=& f\nabla_aT^{ab} Q(L, e_b)+(n-r)f P_{r-1, 0}\left\langle L, \frac{\partial}{\partial t}\right\rangle \\
& + \frac{r}{2}f P_{r, 0} Q(L, \underline{L}) +Q(L, T(\nabla f))\\
\le & (n-r)f P_{r-1, 0}\left\langle L, \frac{\partial}{\partial t}\right\rangle + \frac{r}{2}f P_{r, 0} Q(L, \underline{L}) +Q(L, T(\nabla f)).
\end{aligned}
\end{equation*}
Integrating, we get
\begin{equation*}
\int_{\Sigma} f H_{r-1, 0}\left\langle L, \frac{\partial}{\partial t}\right\rangle d \mu+\frac{1}{2} \int_{\Sigma} f H_{r, 0} Q(L, \underline{L}) d \mu-\frac{1}{ \binom{n-1}{r-1}(n-r)} \int_{\Sigma} Q\left(T_{r, 0}(\nabla f), L\right) d \mu\ge 0.
\end{equation*}

The second inequality can be obtained similarly.
\end{proof}

\begin{theorem}\label{thm schwarzschil alex}
Let $\Sigma$ be a closed embedded spacelike codimension-two submanifold in the Schwarzschild spacetime such that $d \zeta=0$ and $[\zeta]=0 \in H_{d R}^1(\Sigma)$.
Let $L$ and $\underline{L}$ be the null frame given by Proposition \ref{prop torsion}.
\begin{enumerate}
\item
Suppose $\Sigma$ is past incoming null embedded which satisfies the assumptions in either \eqref{(1)} or \eqref{(2)} in Lemma \ref{lem 6.1}, and that
$P_{ r, 0 }(\chi, \underline \chi)>0$ with
\begin{equation*}
\int_{\Sigma} \frac{1}{{P_{r, 0}}^2} Q\left(T_{r, 0}\left(\nabla P_{r, 0}\right), L\right) d \mu \le 0.
\end{equation*}
Then
$\Sigma$ lies in an outgoing null hypersurface of symmetry.
\item
Suppose $\Sigma$ is future incoming null embedded which satisfies the assumptions in either \eqref{(1)} or \eqref{(3)} in Lemma \ref{lem 6.1}, and that $(-1)^s P_{0, s}(\chi, \underline \chi)>0$ with
\begin{equation*}
\int_{\Sigma}\frac{1}{ {P_{0, s}}^{2}} Q\left(T_{0, s}(\nabla P_{0, s}), \underline L\right) d \mu\le 0.
\end{equation*}
Then $\Sigma$ lies in an incoming null hypersurface of symmetry.
\end{enumerate}
\end{theorem}
\begin{proof}
The proof is similar to Theorem \ref{thm higher order alex} and Theorem \ref{thm higher order alex'}, by making use of Theorem \ref{thm schwarz mink} and Lemma \ref{lem newton}.
\end{proof}

\end{document}